\documentclass[11pt]{article}

\usepackage[margin=1in]{geometry}
\usepackage[utf8]{inputenc}
\usepackage[T1]{fontenc}
\usepackage{lmodern}
\usepackage{amsmath,amssymb,amsfonts,amsthm,mathtools}
\usepackage{mathrsfs}
\usepackage{microtype}
\usepackage{enumitem}
\usepackage{booktabs}
\usepackage{tabularx}
\usepackage{longtable}
\usepackage{xcolor}
\usepackage[colorlinks=true,linkcolor=blue,citecolor=blue,urlcolor=blue]{hyperref}
\usepackage[capitalize,nameinlink]{cleveref}

\allowdisplaybreaks
\numberwithin{equation}{section}
\setlist[itemize]{topsep=3pt,itemsep=2pt,parsep=0pt}
\setlist[enumerate]{topsep=3pt,itemsep=2pt,parsep=0pt}

\newtheorem{theorem}{Theorem}[section]
\newtheorem{proposition}[theorem]{Proposition}

\newtheorem{definition}[theorem]{Definition}
\newtheorem{assumption}[theorem]{Assumption}
\newtheorem{conjecture}[theorem]{Conjecture}
\newtheorem{problem}[theorem]{Problem}

\theoremstyle{remark}
\newtheorem{remark}[theorem]{Remark}

\crefname{theorem}{Theorem}{Theorems}
\Crefname{theorem}{Theorem}{Theorems}
\crefname{proposition}{Proposition}{Propositions}
\Crefname{proposition}{Proposition}{Propositions}
\crefname{lemma}{Lemma}{Lemmas}
\Crefname{lemma}{Lemma}{Lemmas}
\crefname{corollary}{Corollary}{Corollaries}
\Crefname{corollary}{Corollary}{Corollaries}
\crefname{definition}{Definition}{Definitions}
\Crefname{definition}{Definition}{Definitions}
\crefname{assumption}{Assumption}{Assumptions}
\Crefname{assumption}{Assumption}{Assumptions}
\crefname{conjecture}{Conjecture}{Conjectures}
\Crefname{conjecture}{Conjecture}{Conjectures}
\crefname{problem}{Problem}{Problems}
\Crefname{problem}{Problem}{Problems}
\crefname{principle}{Principle}{Principles}
\Crefname{principle}{Principle}{Principles}

\DeclareMathOperator*{\esssup}{ess\,sup}
\DeclareMathOperator{\dist}{dist}

\newcommand{\R}{\mathbb R}

\newcommand{\eps}{\varepsilon}
\newcommand{\CKN}{\mathrm{CKN}}

\newcommand{\loc}{\mathrm{loc}}
\newcommand{\dx}{\,dx}

\newcommand{\dxdt}{\,dx\,dt}
\newcommand{\norm}[2]{\left\|#1\right\|_{#2}}
\newcommand{\abs}[1]{\left|#1\right|}
\newcommand{\ip}[2]{\left\langle #1,#2\right\rangle}
\newcommand{\Sup}{\mathrm{Sup}}
\newcommand{\Tax}{\mathrm{Tax}}
\newcommand{\Leak}{\mathrm{Leak}}
\newcommand{\Rep}{\mathrm{Rep}}
\newcommand{\Err}{\mathrm{Err}}

\newcommand{\calB}{\mathcal B}

\newcommand{\calE}{\mathcal E}
\newcommand{\calF}{\mathcal F}

\newcommand{\calL}{\mathcal L}
\newcommand{\calM}{\mathcal M}

\newcommand{\calP}{\mathcal P}

\newcommand{\calR}{\mathcal R}
\newcommand{\calS}{\mathcal S}
\newcommand{\calT}{\mathcal T}

\title{\textbf{A Structural Audit of Navier--Stokes Obstruction Calculus}}

\author{Runlong Yu\\
The University of Alabama, Tuscaloosa, AL, USA\\
\texttt{ryu5@ua.edu}}

\date{}

\begin{document}
\maketitle

\begin{abstract}
	We audit a finite-scale program for the local regularity problem of the three-dimensional incompressible Navier--Stokes equations. The program develops critical ledgers, coarse-grained defect decompositions, pressure--flux work identities, quotient cleanings, and bad-scale counting mechanisms. These results form an obstruction calculus: they locate how Caffarelli--Kohn--Nirenberg badness may be transported, hidden, or reproduced across scales, but they do not by themselves provide a coercive estimate excluding a surviving obstruction. We prove a resolution lemma separating full CKN badness into coarse badness and subfilter residual, and show that no unconditional single-scale domination by a signed combined-work detector is available. The audit therefore identifies the next necessary target: a filtered stretching--diffusion estimate, including subgrid forcing, leakage, pressure tails, and direction-incoherence defects, capable of converting the existing decomposition theory into a genuine regularity or obstruction-exclusion mechanism.
\end{abstract}



\tableofcontents

\section{Introduction: why a structural audit is needed}

\subsection{The local regularity question before smallness}

Let
\[
Q_r(z_0)=B_r(x_0)\times(t_0-r^2,t_0),
\qquad z_0=(x_0,t_0),
\]
and consider the three-dimensional incompressible Navier--Stokes equations
\begin{equation}\label{eq:NS}
\partial_tu-\Delta u+(u\cdot\nabla)u+\nabla p=0,
\qquad \nabla\cdot u=0.
\end{equation}
The local suitable-weak-solution theory descending from Leray and Hopf and developed by Scheffer and Caffarelli--Kohn--Nirenberg turns smallness of scale-critical quantities into local regularity; see \cite{Leray1934,Hopf1951,Scheffer1976,CKN1982,Lin1998,ESS2003,Seregin2015}. A possible singular point must therefore avoid every sufficiently small CKN decay scale.

The program audited here began from a question that precedes the endpoint regularity criterion:
\begin{equation}\label{eq:guiding-question}
\boxed{
\text{If CKN badness persists, what pays for it, hides it, transfers it, or reproduces it?}
}
\end{equation}
This question led to finite-scale ledgers, coarse-grained defect packages, pressure--flux--energy--trace observations, quotient cleanings, reproduction residuals, and finite-chain audit inequalities \cite{YuIDC2026,YuLedger2026,YuAudit2026,YuComp2026,YuLocalClean2026,YuRecursive2026}. More recent modules close two specific interfaces: a conservative finite-chain CKN counting result through one-component compactness \cite{YuCKNCount2026}, and a fixed-chain combined pressure--flux work depletion theorem. A separate resolution lemma isolates the subfilter obstruction to any detector-to-CKN bridge.

The accumulation of these structures creates a new question. Why do many correct finite-window theorems not yet produce a regularity argument? The answer is not that all preceding definitions were empty. Rather, most of them are \emph{obstruction-organizing structures}. They identify where a proof must pay, but they do not themselves prove that the price is unaffordable.

\subsection{Central thesis}

The central conclusion of this article is
\begin{equation}\label{eq:central-thesis}
\boxed{
\text{The existing structure is a successful obstruction calculus, not yet a regularity mechanism.}
}
\end{equation}
The finite-scale ledger is a genuine PDE theorem. The coarse package records genuine Navier--Stokes identities. Fixed-window quotient gaps are genuine finite-dimensional compactness theorems. Local-to-clean and recursive results correctly expose pressure, localization, truncation, synchronization, and reproduction costs. Nevertheless, the surviving obstruction is not excluded because the program still lacks a scale-uniform coercive estimate or a compactness-rigidity theorem.

A useful distinction throughout the paper is
\begin{equation}\label{eq:bookkeeping-coercivity}
\text{bookkeeping or interface structure}
\quad\neq\quad
\text{coercive PDE estimate}.
\end{equation}
A bookkeeping theorem says that every loss appears in a named ledger. A coercive estimate says that the total loss is small enough to absorb, or that a nonzero obstruction cannot satisfy the equation. The first task has been developed substantially. The second remains open in the general three-dimensional setting.

\subsection{Why the next estimate should target vortex stretching}

The audit changes the natural next question. A one-component route asks whether
\[
u_3\text{-degeneration or }u_3\text{-smallness}
\quad\Longrightarrow\quad
\text{regularity}.
\]
This is a meaningful regularity criterion and is closely related to the one-component regularity literature \cite{KukavicaZiane2006,CheminZhang2016,HanLeiLiZhao2019,BarkerPrange2021}. It is not, however, the intrinsic three-dimensional blow-up mechanism. The genuinely dangerous term in the vorticity equation is vortex stretching. For smooth solutions,
\[
\partial_t\omega-\Delta\omega+(u\cdot\nabla)\omega
=(\omega\cdot\nabla)u,
\qquad
(\omega\cdot\nabla)u\cdot\omega=S\omega\cdot\omega,
\]
where $S=(\nabla u+\nabla u^T)/2$. A singular branch, if it exists, must repeatedly rebuild scale-critical badness through a mechanism that keeps positive stretching from being absorbed by diffusion, pressure compatibility, or geometric decorrelation.

Recent Euler blow-up constructions sharpen this lesson. Elgindi proved finite-time singularity formation for $C^{1,\alpha}$ solutions of the incompressible Euler equations on $\R^3$ \cite{Elgindi2021}. Chen and Hou developed a stable nearly self-similar blow-up route for the two-dimensional Boussinesq and three-dimensional axisymmetric Euler setting \cite{ChenHou2023Analysis,ChenHou2023Numerics}. C\'ordoba, Martinez-Zoroa, and Zheng constructed a different non-self-similar mechanism built from infinitely many vorticity regions separated by vortex-free gaps \cite{CMZZ2025}. These mechanisms suggest that the Navier--Stokes problem should be interrogated at the level of recurring stretching cascades, not only through low-dimensional velocity components.

The corrected target is therefore
\begin{equation}\label{eq:corrected-target}
\boxed{
\begin{gathered}
\text{If a vortex-stretching cascade occurs,}\\
\text{how must diffusion, pressure, subgrid transfer,}\\
\text{or vorticity-direction incoherence obstruct it?}
\end{gathered}}
\end{equation}
This is the mechanism-level replacement for the earlier expectation that a power-type one-component approximation estimate would be the central missing lemma.

\subsection{Contributions of this audit}

The paper has six purposes.
\begin{enumerate}
\item We unify the already established finite-scale and finite-window conclusions in a common notation and distinguish theorem-level PDE content from algebraic or conditional interfaces.
\item We identify the exact point at which the current detector route loses contact with full CKN badness.
\item We prove a gauge-invariant resolution lemma and a direct single-scale obstruction to unconditional detector domination.
\item We explain why the main route should shift from one-component power approximation and abstract detector badness to a filtered vortex-stretching--diffusion mechanism.
\item We formulate a small list of new PDE estimates that would genuinely close the existing architecture.
\item We prove an abstract conditional closure theorem showing how defect extraction, silent-channel control, leakage absorption, and backscatter control would imply the occurrence of a CKN-small scale.
\end{enumerate}

\subsection{Theorem-status map}

The following table summarizes the logical status of the principal modules.

\begin{center}
\small
\begin{tabularx}{\textwidth}{@{}p{0.24\textwidth}p{0.20\textwidth}X@{}}
\toprule
Module & Status & Mathematical content and remaining boundary\\
\midrule
Finite-scale critical ledger & Proved PDE theorem & Persistent reservoir badness forces cumulative untaxed supply or leakage on a finite chain. It does not prove uniform taxation.\\
NS-generated coarse package & Proved PDE identities & Momentum, pressure compatibility, covariance positivity, and resolved energy-flux identity hold at fixed coarse length.\\
Clean finite-window quotient gap & Proved under kernel-free hypothesis & Qualitative kernel-freeness is equivalent to a positive finite-dimensional gap; no scale-uniform lower bound follows.\\
Local-to-clean transfer & Conditional algebraic theorem & A clean gap transfers after quotient lifting, detector comparison, and normalized residual absorption. These PDE-facing comparisons remain class-dependent.\\
Recursive audit chains & Proved finite-chain propagation under one-step admissibility & Errors propagate by a variable-coefficient recursion. Infinite-chain summability and uniform constants are not proved.\\
CKN-bad scale counting & Proved with one-component closure & A non-tautological counting theorem follows from classical compactness and the regularity of the $u_3=0$ limit class. It does not close a pure pressure--flux detector.\\
Combined pressure--flux work depletion & Proved fixed-chain PDE theorem & Forward combined work and dissipation are paid by initial energy, leakage, and backscatter. It does not control the sign branch, leakage sum, or moving-window constants.\\
Direct detector-to-CKN domination & False without extra assumptions & A positive CKN norm cannot be universally dominated by a signed work distribution.\\
Filtered stretching--diffusion route & New PDE target with a fixed-$\ell$ identity & The localized filtered enstrophy identity exposes positive stretching against vorticity diffusion and subgrid vorticity forcing. At suitable-weak level it must be formulated with coarse vorticity, not with unfiltered $\nabla\omega$.\\
Scale-uniform compactness-rigidity & Open & The program has not extracted and excluded a canonical minimal obstruction across arbitrarily small scales.\\
\bottomrule
\end{tabularx}
\end{center}

\subsection{Organization}

\Cref{sec:preliminaries} fixes the suitable-weak and CKN setting. \Cref{sec:ledger} records the finite-scale PDE spine. \Cref{sec:packages} describes the NS-generated coarse package. \Cref{sec:finite-window} audits the finite-window quotient geometry and local-to-clean transfer. \Cref{sec:recursive} discusses recursive audit chains and gluing. \Cref{sec:bottleneck} proves the resolution lemma, gives the direct-domination obstruction, and explains the correct silent-mechanism dichotomy. \Cref{sec:wall} explains why the existing structures reach a wall. \Cref{sec:new-estimates} formulates the new PDE estimates. \Cref{sec:stretching-route} gives the corrected filtered vortex-stretching route and the associated bad-scale counting target. \Cref{sec:mechanism} proposes the shift from detector geometry to mechanism geometry. \Cref{sec:conditional-closure} proves a conditional closure theorem. \Cref{sec:conclusion-outlook} closes the paper with a concise conclusion and outlook.

\section{Suitable weak solutions and scale-critical badness}\label{sec:preliminaries}

\subsection{Suitable weak solutions}

\begin{definition}[Suitable weak solution]\label{def:suitable}
Let $Q\subset\R^3\times\R$ be a parabolic cylinder. A pair $(u,p)$ is a suitable weak solution of \eqref{eq:NS} in $Q$ if
\[
u\in L^\infty_tL^2_x(Q)\cap L^2_tH^1_x(Q),
\qquad p\in L^{3/2}(Q),
\]
the equations hold distributionally, $\nabla\cdot u=0$, and for every nonnegative $\phi\in C_c^\infty(Q)$ the local energy inequality holds in the standard almost-everywhere time sense:
\begin{align}\label{eq:LEI}
\int |u(t)|^2\phi(t)\dx
+2\int_{-\infty}^t\!\int |\nabla u|^2\phi\dx\,ds
&\le
\int_{-\infty}^t\!\int |u|^2(\partial_s\phi+\Delta\phi)\dx\,ds\\
&\quad+
\int_{-\infty}^t\!\int (|u|^2+2p)u\cdot\nabla\phi\dx\,ds.
\end{align}
\end{definition}

The pressure is determined only up to an additive function of time. In local pressure decompositions one writes a source-generated Calderon--Zygmund part plus a spatially harmonic remainder; this local-pressure viewpoint is consistent with the pressure-projection literature \cite{Wolf2017}. The harmonic remainder is not a genuine gauge for the momentum equation; its gradient and pressure work remain physical. This distinction is essential in the coarse-package and work-ledger formulations below.

\subsection{Scale-critical quantities}

For $Q_r(z_0)$ define
\begin{align}
A(z_0,r)&=r^{-1}\esssup_{t_0-r^2<t<t_0}\int_{B_r(x_0)}|u(x,t)|^2\dx,\\
E(z_0,r)&=r^{-1}\int_{Q_r(z_0)}|\nabla u|^2\dxdt,\\
C(z_0,r)&=r^{-2}\int_{Q_r(z_0)}|u|^3\dxdt,\\
D(z_0,r)&=r^{-2}\int_{Q_r(z_0)}
\abs{p-(p)_{B_r(x_0)}(t)}^{3/2}\dxdt.
\end{align}
Set
\begin{equation}\label{eq:PsiPhi}
\Phi(z_0,r)=A(z_0,r)+E(z_0,r)+C(z_0,r)+D(z_0,r),
\qquad
\Psi(z_0,r)=C(z_0,r)+D(z_0,r).
\end{equation}
The CKN epsilon-regularity criterion supplies a universal $\eps_{\CKN}>0$ such that sufficiently small $\Psi(z_0,r)$ implies regularity in a smaller cylinder \cite{CKN1982,Lin1998,Seregin2015}.

\begin{definition}[Non-CKN branch]
A point $z_0$ has a non-CKN branch along radii $r_k\downarrow0$ if
\[
\Psi(z_0,r_k)\ge \eps_{\CKN}
\]
for all sufficiently large $k$.
\end{definition}

\subsection{Scaling}

For a radius $r>0$ define
\begin{equation}\label{eq:NS-scaling}
u^{(r)}(y,s)=r u(x_0+ry,t_0+r^2s),
\qquad
p^{(r)}(y,s)=r^2p(x_0+ry,t_0+r^2s).
\end{equation}
Then $(u^{(r)},p^{(r)})$ is suitable on the rescaled cylinder and the quantities $A,E,C,D$ are invariant. This converts the small-scale problem into a sequence of unit-scale states and motivates the language of a renormalized orbit in a critical state space.

\section{The established PDE spine: finite-scale critical ledgers}\label{sec:ledger}

\subsection{Reservoir and transition coordinates}

Fix $0<\theta<1$ and an admissible nested chain
\[
r_k=\theta^kr_0,
\qquad
Q_k=Q_{r_k}(z_k),
\qquad
Q_{k+1}\subset Q_k,
\]
with controlled parabolic drift of the centers. Let $\phi_k\in C_c^\infty(Q_k)$ be nonnegative, equal to one on $Q_{k+1}$, and satisfy the scale-adapted derivative bounds
\[
|\nabla\phi_k|\le C_\theta r_k^{-1},
\qquad
|\partial_t\phi_k|+|\Delta\phi_k|\le C_\theta r_k^{-2}.
\]
Write $A_k=A(z_k,r_k)$ and similarly for $E_k,C_k,D_k$. The reservoir badness is
\begin{equation}\label{eq:reservoir}
B_k=A_k+C_k+D_k.
\end{equation}
The dissipation $E_k$ is kept mainly as a tax rather than as a reservoir coordinate.

Define the transition quantities
\begin{align}
\calF_k&=r_k^{-1}\int_{Q_k}|u|^2|u\cdot\nabla\phi_k|\dxdt,\label{eq:flux-supply}\\
\calP_k&=r_k^{-1}\int_{Q_k}
\abs{p-(p)_{B_{r_k}(x_k)}(t)}|u\cdot\nabla\phi_k|\dxdt,\label{eq:pressure-supply}\\
\calL_k&=r_k^{-1}\int_{Q_k}|u|^2
\bigl(|\partial_t\phi_k|+|\Delta\phi_k|\bigr)\dxdt.\label{eq:cutoff-leak}
\end{align}
The natural state is therefore a point-edge object
\begin{equation}\label{eq:point-edge}
X_k=(A_k,C_k,D_k;E_{k+1};\calF_k,\calP_k,\calL_k).
\end{equation}
This separation is structural: reservoir coordinates live at a scale, whereas flux, pressure transport, and localization live on a transition.

\subsection{One-step ledger}

The local energy inequality gives
\begin{equation}\label{eq:energy-transition}
A_{k+1}+2E_{k+1}
\le \theta^{-1}(\calL_k+\calF_k+2\calP_k).
\end{equation}
The standard interpolation and pressure-decay estimates give
\begin{align}
C_{k+1}&\le C_{I,\theta}
\bigl((\calF_k+2\calP_k)^{3/2}+\calL_k^{3/2}\bigr),\label{eq:interpolation-ledger}\\
D_{k+1}&\le C_P\theta D_k+C_P\theta^{-2}C_k.\label{eq:pressure-decay}
\end{align}
Fix $0<\alpha<1$ and choose $\theta$ so that
\[
\delta_D=(1-\alpha)-C_P\theta>0.
\]
Set
\begin{align}
\Sup_k^{\rm full}
&=\theta^{-1}(\calF_k+2\calP_k)
+C_{I,\theta}(\calF_k+2\calP_k)^{3/2}
+C_P\theta^{-2}C_k,\\
\Tax_k^{\rm full}
&=2E_{k+1}+(1-\alpha)A_k+(1-\alpha)C_k+\delta_DD_k,\\
\Leak_k^{\rm full}
&=\theta^{-1}\calL_k+C_{I,\theta}\calL_k^{3/2}.
\end{align}

\begin{theorem}[Full critical ledger inequality]\label{thm:full-ledger}
Along every admissible finite chain,
\begin{equation}\label{eq:full-ledger}
B_{k+1}-(1-\alpha)B_k
\le
\Sup_k^{\rm full}-\Tax_k^{\rm full}+\Leak_k^{\rm full}.
\end{equation}
\end{theorem}

\begin{proof}
Add \eqref{eq:energy-transition}, \eqref{eq:interpolation-ledger}, and \eqref{eq:pressure-decay}, then subtract $(1-\alpha)(A_k+C_k+D_k)$. The coefficient of $D_k$ is $-\delta_D$, and the remaining terms are precisely those grouped into supply, tax, and leakage.
\end{proof}

\subsection{Finite-scale survival alternative}

\begin{theorem}[Finite-scale survival alternative]\label{thm:survival}
Assume
\[
B_k\ge\eps,
\qquad 0\le k\le N-1.
\]
Then
\begin{equation}\label{eq:survival}
\sum_{k=0}^{N-1}
\bigl(\Sup_k^{\rm full}-\Tax_k^{\rm full}\bigr)_+
\ge
\alpha\eps N-B_0-
\sum_{k=0}^{N-1}\Leak_k^{\rm full}.
\end{equation}
\end{theorem}

\begin{proof}
Set $M_k=B_{k+1}-(1-\alpha)B_k$. Summing gives
\[
\sum_{k=0}^{N-1}M_k
=B_N-B_0+\alpha\sum_{k=0}^{N-1}B_k
\ge -B_0+\alpha\eps N.
\]
On the other hand, \cref{thm:full-ledger} gives
\[
\sum_{k=0}^{N-1}M_k
\le
\sum_{k=0}^{N-1}
(\Sup_k^{\rm full}-\Tax_k^{\rm full})
+
\sum_{k=0}^{N-1}\Leak_k^{\rm full}.
\]
Replace each signed supply--tax difference by its positive part and rearrange.
\end{proof}

The theorem proves the first robust principle of the program:
\begin{equation}\label{eq:not-free}
\boxed{\text{Persistent scale-critical badness is not free.}}
\end{equation}
It does not prove regularity, because it does not show that the untaxed supply is impossible or that leakage has a summable average. The ledger is an accounting theorem, not yet a Lyapunov functional.

\section{NS-generated defect packages}\label{sec:packages}

\subsection{Coarse graining and Reynolds covariance}

Coarse-grained energy flux and local energy-defect formalisms are standard tools in turbulence and weak-solution analysis; the notation below follows this tradition while being restricted to a fixed finite scale \cite{ConstantinETiti1994,Eyink1994,DuchonRobert2000,LeslieShvydkoy2018}.

Let $S_\ell f=\rho_\ell*_xf$ be spatial convolution with a nonnegative mollifier at length $\ell>0$. On an interior cylinder define
\begin{equation}\label{eq:coarse-package}
U^\ell=S_\ell u,
\qquad
P^\ell=S_\ell p,
\qquad
R^\ell=S_\ell(u\otimes u)-U^\ell\otimes U^\ell.
\end{equation}
Because $S_\ell$ is an averaging operator,
\begin{equation}\label{eq:R-positive}
R^\ell=(R^\ell)^T\ge0
\end{equation}
as a quadratic form. The resolved flux is
\begin{equation}\label{eq:Pi-def}
\Pi^\ell=-R^\ell:\nabla U^\ell.
\end{equation}

\begin{proposition}[NS-generated coarse package]\label{prop:coarse-package}
The fields in \eqref{eq:coarse-package} satisfy, on every interior cylinder,
\begin{align}
\partial_tU^\ell-\Delta U^\ell
+\nabla\cdot(U^\ell\otimes U^\ell)+\nabla P^\ell
&=-\nabla\cdot R^\ell,\label{eq:coarse-momentum}\\
\nabla\cdot U^\ell&=0,\label{eq:coarse-div}\\
-\Delta P^\ell&=\partial_i\partial_j
(U_i^\ell U_j^\ell+R_{ij}^\ell),\label{eq:coarse-pressure}\\
\partial_t\frac{|U^\ell|^2}{2}
-\Delta\frac{|U^\ell|^2}{2}
+|\nabla U^\ell|^2
+\nabla\cdot\left[
\left(\frac{|U^\ell|^2}{2}+P^\ell\right)U^\ell+R^\ell U^\ell
\right]
&=-\Pi^\ell.\label{eq:coarse-energy}
\end{align}
\end{proposition}

\begin{proof}
Spatially convolve \eqref{eq:NS}. The difference between $S_\ell(u\otimes u)$ and $U^\ell\otimes U^\ell$ is $R^\ell$, which gives \eqref{eq:coarse-momentum}. Taking divergence gives \eqref{eq:coarse-pressure}. Dotting \eqref{eq:coarse-momentum} with $U^\ell$ and integrating by parts algebraically gives \eqref{eq:coarse-energy}. Positivity follows from Jensen's inequality:
\[
\xi\cdot R^\ell\xi
=S_\ell((\xi\cdot u)^2)-(S_\ell(\xi\cdot u))^2\ge0.
\]
\end{proof}

\subsection{Active and harmonic pressure}

In a local pressure window choose a cutoff $\eta$ equal to one on the observation core and set schematically
\begin{equation}\label{eq:active-pressure}
P^{\rm act}=R_iR_j\bigl(\eta(U_i^\ell U_j^\ell+R_{ij}^\ell)\bigr),
\qquad
P^{\rm har}=P^\ell-P^{\rm act}.
\end{equation}
Then $P^{\rm har}$ is spatially harmonic on the core. Only additions $a(t)$ are genuine pressure gauges for the momentum equation. A spatially harmonic pressure is retained unless a particular quotient theorem places a controlled finite-dimensional part of it into a cleaning space and charges the tail explicitly.

\subsection{What the package accomplishes}

The package
\begin{equation}\label{eq:D-package}
D_{k,\ell}=(U^\ell,P^\ell;P^{\rm act},P^{\rm har},R^\ell,\Pi^\ell)
\end{equation}
is stronger than an arbitrary algebraic defect. Its coordinates satisfy the same equation, pressure compatibility, covariance positivity, and energy-transfer identity. This is the content of the \emph{NS-realizability filter}: a formal kernel vector is relevant only if it lies in the closure of packages generated by suitable weak solutions.

At the same time, the package is not a normal form for a singularity. It is a collection of compatible coordinates at a chosen scale and coarse length. Nothing in \cref{prop:coarse-package} implies that the package is nontrivial whenever CKN badness is nontrivial, nor that it reproduces uniformly across scales.

\section{Finite-window obstruction geometry}\label{sec:finite-window}

\subsection{Cleaning, quotient, and observation}

Fix a finite scale or frequency window $\Lambda$. Let $D_\Lambda$ be a finite-dimensional constrained defect space, $C_\Lambda$ a cleaning space, and
\[
G_\Lambda:C_\Lambda\to D_\Lambda
\]
a linear cleaning map. The quotient distance is
\begin{equation}\label{eq:quot-distance}
\dist_\Lambda(d,\operatorname{Im}G_\Lambda)
=
\inf_{c\in C_\Lambda}\norm{d-G_\Lambda c}{D_\Lambda}.
\end{equation}
Let
\[
O_\Lambda:D_\Lambda\to Y_\Lambda
\]
collect active pressure, flux, energy, and trace observations. In a dynamic version, one also has a reproduction residual $\Rep_\Lambda(d)$ and a nonnegative tax or profit-cost functional $\Tax_\Lambda(d)$.

\subsection{Exact finite-window quotient theorem}

\begin{theorem}[Finite-window anti-phantom gap]\label{thm:finite-gap}
Assume the detector
\begin{equation}\label{eq:detector-size}
\calM_\Lambda(d)
=\norm{O_\Lambda d}{Y_\Lambda}
+C_R\Rep_\Lambda(d)+C_T\Tax_\Lambda(d)
\end{equation}
is continuous, positively homogeneous, and constant on gauge cosets. Assume further that
\begin{equation}\label{eq:kernel-free}
O_\Lambda d=0,
\quad
\Rep_\Lambda(d)=0,
\quad
\Tax_\Lambda(d)=0
\quad\Longrightarrow\quad
d\in\operatorname{Im}G_\Lambda.
\end{equation}
Then there exists $\mu_\Lambda>0$ such that
\begin{equation}\label{eq:finite-gap}
\calM_\Lambda(d)
\ge
\mu_\Lambda\dist_\Lambda(d,\operatorname{Im}G_\Lambda)
\qquad
\text{for all }d\in D_\Lambda.
\end{equation}
\end{theorem}

\begin{proof}
The detector descends to a continuous positively homogeneous function on the finite-dimensional quotient $D_\Lambda/\operatorname{Im}G_\Lambda$. By \eqref{eq:kernel-free}, it is strictly positive on the quotient unit sphere. Compactness of that sphere gives a positive minimum $\mu_\Lambda$. This is a finite-dimensional compactness step; it should not be confused with scale-uniform PDE compactness such as the Aubin--Lions--Simon mechanism \cite{Simon1986}. Homogeneity gives \eqref{eq:finite-gap}.
\end{proof}

This theorem is rigorous but limited. The substantive hypothesis is \eqref{eq:kernel-free}; finite-dimensional compactness merely turns qualitative kernel-freeness into a quantitative gap. The constant may collapse as the window grows.

\subsection{Local-to-clean transfer}

Let $D_\Lambda^{\loc}$ and $D_\Lambda^{\rm cl}$ be localized and clean defect spaces, with cleaning images $\Gamma_\Lambda^{\loc}$ and $\Gamma_\Lambda^{\rm cl}$. Let
\[
\Theta_\Lambda:D_\Lambda^{\loc}\to D_\Lambda^{\rm cl}
\]
be a local-to-clean chart. Suppose
\begin{align}
\dist_{\rm cl}(\Theta_\Lambda D,\Gamma_\Lambda^{\rm cl})
&\ge (1-\eps_G)\dist_{\loc}(D,\Gamma_\Lambda^{\loc})-\delta_G,\label{eq:quot-lift}\\
\calM_\Lambda^{\loc}(D)
&\ge \calM_\Lambda^{\rm cl}(\Theta_\Lambda D)-\Err_\Lambda(D),\label{eq:det-compare}\\
\Err_\Lambda(D)
&\le \eta_\Lambda\dist_{\loc}(D,\Gamma_\Lambda^{\loc})+\Delta_\Lambda.
\label{eq:err-budget}
\end{align}

\begin{theorem}[Conditional local-to-clean transfer]\label{thm:local-clean}
If the clean detector has gap $\mu_\Lambda$ and
\[
\eta_\Lambda<\mu_\Lambda(1-\eps_G),
\]
then
\begin{equation}\label{eq:local-clean}
\calM_\Lambda^{\loc}(D)
\ge
\bigl(\mu_\Lambda(1-\eps_G)-\eta_\Lambda\bigr)
\dist_{\loc}(D,\Gamma_\Lambda^{\loc})
-
\mu_\Lambda\delta_G-\Delta_\Lambda.
\end{equation}
\end{theorem}

\begin{proof}
Apply the clean gap to $\Theta_\Lambda D$, then use \eqref{eq:quot-lift}, \eqref{eq:det-compare}, and \eqref{eq:err-budget} in that order.
\end{proof}

The theorem correctly isolates the residuals that must be paid: pressure tails, cutoff leakage, projection and truncation loss, nonlinear mismatch, reproduction drift, gauge synchronization, and tax discrepancy. Yet it is an interface theorem. The hard PDE question is whether the coefficient $\eta_\Lambda$ is below the clean gap margin and whether the additive errors are summable or negligible along a singular branch.

\subsection{What is closed and what remains conditional}

The finite-window calculus has genuinely closed the following issues:
\begin{itemize}
\item the need to evaluate all residuals on synchronized representatives;
\item finite-dimensional pressure-tail approximation and harmonic polynomial tails;
\item componentwise residual-ledger assembly;
\item clean quotient compactness and singular-value criteria;
\item algebraic transfer once the normalized error bounds are available.
\end{itemize}
It has not proved, for arbitrary suitable weak solutions and moving windows:
\begin{itemize}
\item a scale-uniform clean gap;
\item an intrinsic chart lower bound;
\item a uniform component-to-baseline comparison;
\item summable pressure, localization, and reproduction errors;
\item extraction of a nonzero finite-window defect from every non-CKN scale.
\end{itemize}

\section{Recursive audit chains and the gluing boundary}\label{sec:recursive}

\subsection{One-step admissibility}

A recursive package map has the schematic form
\begin{equation}\label{eq:one-step-map}
D_k\longmapsto D_{k+1}=\calR_kD_k,
\end{equation}
where the operation consists of restriction, Navier--Stokes rescaling, representative synchronization, and recomputation of pressure, source, localization, gate, detector, residual, and quotient coordinates. Let $E_k$ be the accumulated audit error and $\Delta_k$ the one-step increment.

\begin{assumption}[Recursive error inequality]\label{ass:recursive-error}
There exist coefficients $a_k,b_k\ge0$ such that
\begin{equation}\label{eq:error-recursion}
E_{k+1}\le a_kE_k+b_k\Delta_k.
\end{equation}
\end{assumption}

Suppose a static audit certificate at scale $k$ has the form
\begin{equation}\label{eq:static-cert}
M_k\ge c_k\delta_k-E_k,
\end{equation}
where $\delta_k$ is a baseline defect distance and $M_k$ is a localized detection size.

\begin{proposition}[Finite-chain recursive lower bound]\label{prop:recursive-sum}
Under \eqref{eq:error-recursion} and \eqref{eq:static-cert}, for every nonnegative weights $w_k$,
\begin{equation}\label{eq:recursive-sum}
\sum_{k=0}^Kw_kM_k
\ge
c_K^{\min}\sum_{k=0}^Kw_k\delta_k
-
E_K^{\rm rec},
\end{equation}
where $c_K^{\min}=\min_{0\le k\le K}c_k$ and $E_K^{\rm rec}$ is the explicit weighted sum obtained by iterating \eqref{eq:error-recursion}.
\end{proposition}

\begin{proof}
Iterate \eqref{eq:error-recursion} to express each $E_k$ as the propagated initial error plus a sum of the increments $\Delta_j$ multiplied by products of the coefficients $a_i$ and $b_j$. Insert these bounds into \eqref{eq:static-cert}, multiply by $w_k$, and sum. Replacing $c_k$ by $c_K^{\min}$ gives \eqref{eq:recursive-sum}.
\end{proof}

\subsection{Why finite-chain propagation is not enough}

The proposition has no infinite-scale consequence unless one proves at least:
\begin{equation}\label{eq:recursive-needs}
\inf_k c_k>0,
\qquad
\sum_k\text{propagated }\Delta_k<\infty,
\qquad
\sup_k\prod_{j<k}a_j<\infty
\end{equation}
in an appropriate weighted sense. None of these is a consequence of finite-dimensional compactness alone.

There is also a gluing obstruction. Even if every finite window is cleanable, the cleanings may fail to form a compatible inverse-system section. In algebraic language, finite exactness does not by itself imply exactness after taking an inverse limit; a derived-limit obstruction may remain. In PDE language, this is the possibility that every finite shadow can be corrected, while the correction costs or gauges fail to glue along an infinite scale chain.

Thus recursive audit theory identifies the correct transition problem, but it does not produce a scale-uniform normal form.

\section{The detector-to-CKN bottleneck}\label{sec:bottleneck}

\subsection{Full, coarse, and residual badness}

Fix a spatial coarse length $\ell>0$. Let $U^\ell=S_\ell u$ and $P^\ell=S_\ell p$. On a cylinder $Q_r(z_0)$ define
\begin{align}
\Psi(z_0,r)
&=r^{-2}\int_{Q_r(z_0)}
\left(|u|^3+|\mathbb P_{r,z_0}p|^{3/2}\right)\dxdt,\label{eq:full-badness}\\
\Psi^\ell(z_0,r)
&=r^{-2}\int_{Q_r(z_0)}
\left(|U^\ell|^3+|\mathbb P_{r,z_0}P^\ell|^{3/2}\right)\dxdt,\label{eq:coarse-badness}\\
\Omega^\ell(z_0,r)
&=r^{-2}\int_{Q_r(z_0)}
\left(|u-U^\ell|^3+
|\mathbb P_{r,z_0}(p-P^\ell)|^{3/2}\right)\dxdt,
\label{eq:residual-badness}
\end{align}
where
\[
\mathbb P_{r,z_0}f=f-(f)_{B_r(x_0)}(t).
\]

\subsection{Resolution lemma}

\begin{theorem}[Resolution lemma]\label{thm:resolution}
For every cylinder on which the quantities above are finite,
\begin{equation}\label{eq:resolution}
\Psi(z_0,r)\le4\Psi^\ell(z_0,r)+4\Omega^\ell(z_0,r).
\end{equation}
Consequently, if $\Psi(z_0,r)\ge\eps_0$ and
$\Omega^\ell(z_0,r)\le\eta\eps_0$ for some $0<\eta<1/4$, then
\begin{equation}\label{eq:coarse-visible}
\Psi^\ell(z_0,r)\ge\left(\frac14-\eta\right)\eps_0.
\end{equation}
\end{theorem}

\begin{proof}
Write
\[
u=U^\ell+(u-U^\ell).
\]
The inequality $|a+b|^3\le4(|a|^3+|b|^3)$ gives the velocity bound. Since the pressure projector is linear,
\[
\mathbb P_{r,z_0}p
=
\mathbb P_{r,z_0}P^\ell
+
\mathbb P_{r,z_0}(p-P^\ell).
\]
Using $|a+b|^{3/2}\le2^{1/2}(|a|^{3/2}+|b|^{3/2})$ gives the pressure bound. Divide by $r^2$ and use $2^{1/2}\le4$. Rearranging yields \eqref{eq:coarse-visible}.
\end{proof}

The lemma proves only
\begin{equation}\label{eq:first-bridge}
\text{full CKN badness}
\Longrightarrow
\text{coarse CKN badness or subfilter residual}.
\end{equation}
It does not prove that coarse badness produces detectable work.

\subsection{The signed work detector}

The natural coarse work distribution is
\begin{equation}\label{eq:G}
G^\ell
=
\Pi^\ell+\nabla\cdot(P^\ell U^\ell),
\qquad
\Pi^\ell=-R^\ell:\nabla U^\ell.
\end{equation}
For a test function $\varphi$,
\[
\ip{G^\ell}{\varphi}
=
\int \varphi\Pi^\ell\dxdt
-
\int P^\ell U^\ell\cdot\nabla\varphi\dxdt.
\]
A finite or infinite detector seminorm has the form
\begin{equation}\label{eq:detector-seminorm}
A_r(G^\ell)
=
\sup_{\psi\in\calT_r,\ \norm{\psi}{X_r}\le1}
r^{-1}\abs{\ip{G^\ell}{\chi_r\psi}}.
\end{equation}
The structural mismatch is immediate: $\Psi$ is positive and norm-like, while $G^\ell$ is signed and cancellation-sensitive.

\subsection{No direct single-scale domination}

\begin{proposition}[No unconditional direct detector domination]\label{prop:no-direct}
Let $A_r$ be any seminorm of the form \eqref{eq:detector-seminorm}, so that $A_r(G)=0$ whenever $G=0$. There is no universal constant $C$ such that every smooth local Navier--Stokes solution satisfies
\begin{equation}\label{eq:false-direct}
\Psi(z_0,r)\le C A_r(G^\ell)
\end{equation}
for a fixed $r>0$ and fixed $\ell>0$.
\end{proposition}

\begin{proof}
Let $u(x,t)=a\in\R^3$ be a spatially and temporally constant velocity and let $p=0$. This is a smooth Navier--Stokes solution. Spatial coarse graining leaves $U^\ell=a$, while $R^\ell=0$, $P^\ell=0$, and therefore $G^\ell=0$. Hence $A_r(G^\ell)=0$. On the other hand,
\[
\Psi(z_0,r)
=r^{-2}\int_{Q_r(z_0)}|a|^3\dxdt
=|B_1|\,|a|^3r^3>0
\]
for $a\ne0$. Thus \eqref{eq:false-direct} fails.
\end{proof}

\begin{remark}
The example is regular and its CKN quantity decays as $r\downarrow0$. It does not produce a persistent bad branch. Its role is sharper and more limited: it rules out the proposed \emph{single-scale universal domination}. Any valid bridge must use persistence, decay failure, normalization, dissipation, or an explicit silent alternative.
\end{remark}

\subsection{Fixed-chain combined-work depletion}

At fixed $\ell>0$ and on a finite chain of adjacent slabs, the combined work identity gives selected weights with work $W_k$, resolved dissipation $D_k$, leakage $L_k$, and weights $w_k=r_k/r_0$ satisfying
\begin{equation}\label{eq:work-depletion}
\sum_{k=0}^{N-1}w_k(W_k^++D_k)
\le
\calE_0
+
\sum_{k=0}^{N-1}w_k|L_k|
+
\sum_{k=0}^{N-1}w_kW_k^-.
\end{equation}
Here $\calE_0$ denotes the initial resolved kinetic-energy budget available on the first slab of the fixed chain.
The active extraction constant is the smallest singular value of an explicit finite matrix. The theorem closes a genuine PDE input isolated in the earlier audit program: detected combined work is paid by a finite kinetic-energy ledger. It does not show that a CKN-bad scale is detected, that the forward branch occurs, or that leakage and backscatter are summable.

\subsection{Silent mechanisms and the correct dichotomy}

The direct bridge must be replaced by
\begin{equation}\label{eq:dichotomy}
\boxed{
\text{CKN badness}
\Longrightarrow
\text{detectable coarse work}
\quad\text{or}\quad
\text{an explicit silent mechanism}.
}
\end{equation}
The recurrent silent mechanisms identified by the existing framework are:
\begin{enumerate}
\item \textbf{Subfilter residual:} $\Omega^\ell(r)$ remains non-negligible, so the badness lies below the detector resolution.
\item \textbf{Harmonic pressure tail:} the local harmonic pressure carries work or oscillation not captured by the active pressure projection.
\item \textbf{Pressure--flux cancellation:} pressure and flux channels are individually active while their signed sum is small.
\item \textbf{Coherent low-frequency flow:} a large smooth resolved field is CKN-large at a selected scale but has negligible combined work.
\item \textbf{Backscatter:} the selected work lies on the negative branch and refills rather than depletes the forward ledger.
\item \textbf{Moving-window collapse:} every fixed window is coercive but the constants degenerate along the scale-selected sequence.
\end{enumerate}
These are not cosmetic error terms. Each is a candidate mechanism by which the current detector can fail to be a coordinate for full CKN badness.

\section{Why the existing structures reach a wall}\label{sec:wall}

\subsection{Reduction is not closure}

Many results in the program have the logical form
\begin{equation}\label{eq:interface-form}
\text{structural inputs}
\Longrightarrow
\text{finite-window lower bound}.
\end{equation}
The difficulty has therefore been moved, not eliminated. The structural inputs are often exactly the hard PDE estimates: scale-uniform chart visibility, residual absorption, moving-window control, and NS-realizable kernel exclusion.

\subsection{The ledger is not a Lyapunov function}

The ledger \eqref{eq:full-ledger} has the form
\[
\text{new reservoir}-\text{expected old reservoir}
\le
\text{supply}-\text{tax}+\text{leakage}.
\]
A Lyapunov mechanism would require a theorem such as
\begin{equation}\label{eq:uniform-taxation}
\Sup_k^{\rm full}
\le
\Tax_k^{\rm full}
+\text{summable error}
\end{equation}
for every NS-realizable branch, or an alternative that converts every excess supply into a quantity known to be finite. Such a theorem is not presently available.

\subsection{Fixed-window exactness is not scale-uniform exactness}

The positive finite-window gap $\mu_\Lambda$ is obtained by compactness of a fixed quotient unit sphere. Nothing prevents
\begin{equation}\label{eq:gap-collapse}
\mu_{\Lambda_n}\downarrow0
\end{equation}
as the active window, chain depth, spatial aspect ratio, or pressure-tail resolution changes. Pure gluing observables already exhibit $O(L^{-1})$ singular-value decay on a chain of length $L$, and terminal parabolic observations may be exponentially weak at high frequency. Thus uniform observability is first a compactness-and-separation theorem, not a matrix theorem.

\subsection{Obstruction containers are not normal forms}

The phrase
\[
\text{NS-realizable, cleaned, scale-critical, combined-invisible, profitable, reproducible cascade}
\]
correctly records many necessary properties of a dangerous branch. But it is still an obstruction container: it permits many dynamical geometries. A genuine compactness-rigidity structure would force a minimal obstruction into a substantially narrower class, such as an ancient solution, a scale-stationary orbit, a self-similar profile, or a monotone defect measure. The present framework has not yet produced that normal form.

\subsection{The old observables may be insufficient}

The one-component Schur-visibility analysis gives a useful negative lesson. An abstract envelope containing all previously selected observables can still permit arbitrary-slow branches unless one adds a genuinely Navier--Stokes-specific compatibility channel, such as vertical-pressure visibility or vertical-duality control \cite{YuSchur2026}. This is evidence for a general principle:
\begin{equation}\label{eq:observable-principle}
\boxed{
\text{Adding more observables helps only if the equation makes them coercively compatible.}
}
\end{equation}
A larger list of coordinates without a new compatibility estimate increases bookkeeping without decreasing the obstruction class.

\section{What is missing: the new PDE estimates}\label{sec:new-estimates}

This section states the estimates forced by the preceding audit. They are theorem targets, not claims proved here unless explicitly stated. The list should now be read with a hierarchy: the most concrete first target is not another detector comparison theorem, but a filtered stretching--diffusion estimate that speaks directly to the mechanism by which three-dimensional badness could be rebuilt.

\subsection{A non-tautological defect extraction theorem}

The first missing bridge is from persistent CKN badness to a nonzero object in the chosen obstruction geometry.

\begin{problem}[Defect extraction]\label{prob:defect-extraction}
Assume $\Phi(z_0,r_k)\le M$ and $\Psi(z_0,r_k)\ge\eps_{\CKN}$ along an infinite scale sequence. Construct, after rescaling and subsequence extraction, a normalized NS-realizable obstruction that is nonzero in a topology not containing a disguised copy of the full CKN norm.
\end{problem}

A candidate package may contain
\begin{equation}\label{eq:candidate-defect-package}
\mathfrak D
=(u_*,p_*;\nu_E,\nu_{\rm LEI},[h],\calR,\calT),
\end{equation}
where $\nu_E$ is an energy or dissipation concentration measure, $\nu_{\rm LEI}$ is a local-energy defect, $[h]$ is a retained harmonic-pressure class, $\calR$ is an unresolved residual, and $\calT$ is a scale-transition law. One may also record weak limits of the $L^3$ and $L^{3/2}$ densities when they carry genuine concentration.

A caution is necessary. Under a full local critical bound, standard suitable-weak compactness often gives strong $L^3_{\loc}$ convergence of velocity and, after pressure decomposition, strong interior convergence of mean-free pressure. Therefore velocity and pressure concentration measures may vanish in precisely the regime where one hopes to extract a nontrivial obstruction. The canonical defect may reside not in a positive measure but in scale recurrence, transition failure, harmonic tails, or normalization. The extraction theorem should discover the correct object rather than presuppose that it is measure-valued.

\subsection{A useful subfilter estimate}

The resolution lemma is algebraic. To make it effective one needs a scale-uniform estimate on $\Omega^\ell$.

\begin{problem}[Subfilter compactness]\label{prob:subfilter}
Find a critical modulus $\omega_M$ such that, on an appropriate normalized class,
\begin{equation}\label{eq:subfilter-target}
\Omega^\ell(z_0,r)
\le
\omega_M(\ell/r),
\qquad
\omega_M(s)\to0\quad(s\downarrow0),
\end{equation}
possibly after separating an explicit harmonic pressure tail.
\end{problem}

An elementary starting point is the local increment bound
\begin{equation}\label{eq:increment-target}
\Omega^\ell(z_0,r)
\lesssim
r^{-2}\sup_{|h|\le\ell}
\int_{Q_r(z_0)}|u(x+h,t)-u(x,t)|^3\dxdt
+\text{pressure increment term}.
\end{equation}
A Besov or compactness modulus would turn \eqref{eq:increment-target} into a power of $\ell/r$. The obstacle is that such a modulus is not available for arbitrary suitable weak solutions at the CKN threshold.

\subsection{Caloric leakage absorption}

The localization functional in a resolved energy identity contains
\[
\frac12|U|^2(\partial_t\phi+\Delta\phi),
\qquad
\frac12|U|^2U\cdot\nabla\phi,
\qquad
(RU)\cdot\nabla\phi.
\]
A promising design principle is to choose nonnegative detector weights that are approximately backward caloric on the observation core.

\begin{problem}[Caloric leakage absorption]\label{prob:leakage}
Construct an admissible detector class for which
\begin{equation}\label{eq:leakage-absorption}
\sum_kw_k|L_k|
\le
\eta\sum_kw_k A_k^\infty(G^\ell)
+C_\eta(M)
\end{equation}
with $\eta$ smaller than the detector extraction coefficient, or prove a comparable annular packing estimate.
\end{problem}

The difficulty is structural. Detector richness favors flexible profiles, while leakage control favors caloric profiles. A successful theorem must show that a finite norming family can be chosen inside a low-leakage test class.

\subsection{Backscatter control under geometric hypotheses}

Since
\[
\Pi^\ell=-R^\ell:S(U^\ell),
\qquad
S(U)=\frac12(\nabla U+\nabla U^T),
\]
and $R^\ell\ge0$, the sign of the flux is determined by how the covariance aligns with compressive and expansive strain directions. No unconditional sign is available.

\begin{problem}[Conditional backscatter control]\label{prob:backscatter}
Identify a geometric or recurrence hypothesis under which
\begin{equation}\label{eq:backscatter-target}
\sum_{k\in I_-}w_kW_k^-
\le
\gamma\sum_{k\in I_+}w_kW_k^+
+C(M),
\qquad 0\le\gamma<1.
\end{equation}
\end{problem}

Possible hypotheses include averaged strain alignment, forward sign coherence, restrictions on the density of negative-work slabs, or an enstrophy-production inequality. The theorem should not conceal the sign assumption in the definition of a selected window.

\subsection{Pressure-tail compactness}

Finite-dimensional harmonic polynomial approximation is already available. What is missing is an equation-derived compactness theorem for the moving family of active pressure images and retained harmonic tails.

\begin{problem}[Pressure-tail compactness]\label{prob:pressure-tail}
For a normalized NS-generated class, prove uniform finite-rank approximation of the active pressure image and a summable or decaying bound for the retained harmonic pressure work on nested cores.
\end{problem}

The active part should be controlled by compactness of $u\otimes u$ under a topology strong enough for the Calderon--Zygmund map. The harmonic part should use interior analytic decay, but the time-dependent amplitudes and moving centers must be controlled rather than quotiented away.

\subsection{Bad-but-silent profile classification}

This is the most direct test of whether the present work observable is regularity-relevant.

\begin{problem}[Bad-but-silent compactness classification]\label{prob:silent-classification}
Let $(u^{(n)},p^{(n)})$ be normalized suitable weak solutions with persistent coarse badness and
\[
A_k^\infty(G_n^\ell)\to0.
\]
Classify every subsequential limit satisfying $G_*^\ell=0$. Show that each such limit is either regular/decaying, subfilter-dominated, harmonic-pressure-dominated, cancellation-dominated, coherent low-frequency, or backscatter-dominated.
\end{problem}

A classification theorem is valuable even if some classes cannot be excluded. It reveals whether the detector misses a genuine singular mechanism or merely smooth scale-local amplitude.

\subsection{Filtered stretching--diffusion balance}

The most mechanism-facing new estimate is a stretching--diffusion balance. The tempting smooth-level quantities are
\begin{equation}\label{eq:smooth-stretching-quantities}
\Omega_
u(r)=r^{-1}\iint_{Q_r}|\omega|^2\dxdt,
\qquad
P_
u(r)=r\iint_{Q_r}|\nabla\omega|^2\dxdt,
\qquad
V_
u^+(r)=r\iint_{Q_r}(S\omega\cdot\omega)_+\dxdt.
\end{equation}
These are scale invariant, and they express the correct physical competition. However, $P_u(r)$ is not a legitimate starting quantity for arbitrary suitable weak solutions. At the suitable-weak level one has $\nabla u\in L^2$, hence $\omega\in L^2$, but not in general $\nabla\omega\in L^2$. Therefore the first robust estimate must be formulated after spatial filtering.

\begin{problem}[Filtered stretching--diffusion surplus]\label{prob:filtered-stretching-surplus}
Let $U^\ell=S_\ell u$, $\Omega^\ell=\nabla\times U^\ell$, and $S^\ell=(\nabla U^\ell+(\nabla U^\ell)^T)/2$. Prove, on a normalized suitable-weak class and for relative filter scales $\ell/r$ bounded away from $0$ and $\infty$, an estimate of the schematic form
\begin{equation}\label{eq:filtered-stretching-surplus}
\mathsf V_{r,\ell}^+
\le
(1-\eps_*)\mathsf P_{r,\ell}
+C(M)\mathsf O_{r,\ell}
+C\mathsf A_{r,\ell}
+\mathsf R_{r,\ell}
+\mathsf L_{r,\ell}.
\end{equation}
Here
\[
\mathsf V_{r,\ell}^+=r\iint_{Q_r}\chi_r(S^\ell\Omega^\ell\cdot\Omega^\ell)_+,
\quad
\mathsf P_{r,\ell}=r\iint_{Q_r}\chi_r|\nabla\Omega^\ell|^2,
\quad
\mathsf O_{r,\ell}=r^{-1}\iint_{Q_r}\chi_r|\Omega^\ell|^2,
\]
$\mathsf A_{r,\ell}$ is a vorticity-direction incoherence defect, $\mathsf R_{r,\ell}$ is the subgrid vorticity-transfer defect, and $\mathsf L_{r,\ell}$ is cutoff or caloric leakage.
\end{problem}

The estimate should be interpreted as a finite-scale obstruction to Euler-type stretching cascades: if positive stretching remains large, then either diffusion absorbs it or the branch pays through direction incoherence, unresolved subfilter forcing, leakage, or pressure-compatible transfer. This target is closer to the three-dimensional mechanism than a one-component power approximation estimate.

\subsection{Minimal obstruction rigidity}

The long-term compactness-rigidity target is to minimize a coercive defect size over the nonzero obstruction class.

\begin{problem}[Minimal obstruction rigidity]\label{prob:minimal-rigidity}
Define an admissible obstruction class closed under rescaling and compact limits. Prove existence of a minimal nonzero obstruction $\mathfrak D_{\min}$ and show that minimality forces a rigid transition law, for example
\begin{equation}\label{eq:scale-stationary}
\calS_\theta\mathfrak D_{\min}
\simeq
\mathfrak D_{\min}
\end{equation}
in an appropriate quotient. Then exclude all such rigid objects.
\end{problem}

This is the point at which the framework would become a genuine compactness-rigidity theory. Without a canonical size, compactness, and a strict rigidity theorem, ``minimal obstruction'' remains only terminology.

\section{Filtered vortex-stretching route}\label{sec:stretching-route}

\subsection{Why the unfiltered enstrophy route must be corrected}

A natural first attempt is to compare
\[
V^+(r)=r\iint_{Q_r}(S\omega\cdot\omega)_+\dxdt
\quad\text{and}\quad
P(r)=r\iint_{Q_r}|\nabla\omega|^2\dxdt.
\]
This is the right smooth-level intuition but not the right weak-level formulation. For suitable weak solutions the vorticity belongs to $L^2$ locally, while $\nabla\omega$ need not be locally square integrable. Therefore a theorem stated directly with $P(r)$ either assumes extra regularity or is only an a priori identity for smooth approximants. The robust object is the filtered vorticity at a relative scale $\ell=\sigma r$.

This correction is not cosmetic. It changes the route from a formal enstrophy argument to a genuine coarse-grained mechanism argument. The unresolved stress $R^\ell$ now appears in the vorticity equation and must be charged explicitly. This is exactly where the previous defect calculus can still help: it can organize subfilter transfer, leakage, pressure-tail effects, and moving-window losses, but the main estimate must be a PDE estimate for the filtered vorticity mechanism.

\subsection{Filtered vorticity identity}

Let $S_\ell$ be the spatial mollifier from \cref{sec:packages}. Set
\[
U^\ell=S_\ell u,
\qquad
P^\ell=S_\ell p,
\qquad
R^\ell=S_\ell(u\otimes u)-U^\ell\otimes U^\ell,
\]
and define
\[
\Omega^\ell=\nabla\times U^\ell,
\qquad
S^\ell=\frac12(\nabla U^\ell+(\nabla U^\ell)^T),
\qquad
\mathcal J^\ell=\nabla\times(\nabla\cdot R^\ell).
\]

\begin{proposition}[Filtered vorticity enstrophy identity]\label{prop:filtered-vorticity-identity}
On every interior cylinder and for every fixed $\ell>0$, the filtered vorticity satisfies
\begin{equation}\label{eq:filtered-vorticity-equation}
\partial_t\Omega^\ell-\Delta\Omega^\ell+(U^\ell\cdot\nabla)\Omega^\ell
=(\Omega^\ell\cdot\nabla)U^\ell-\mathcal J^\ell.
\end{equation}
Consequently,
\begin{equation}\label{eq:filtered-enstrophy-identity}
\partial_t\frac{|\Omega^\ell|^2}{2}
-\Delta\frac{|\Omega^\ell|^2}{2}
+U^\ell\cdot\nabla\frac{|\Omega^\ell|^2}{2}
+|\nabla\Omega^\ell|^2
=
S^\ell\Omega^\ell\cdot\Omega^\ell
-\Omega^\ell\cdot\mathcal J^\ell.
\end{equation}
\end{proposition}

\begin{proof}
Taking curl of the filtered momentum equation in \eqref{eq:coarse-momentum} eliminates the pressure and gives \eqref{eq:filtered-vorticity-equation}. Dotting the result with $\Omega^\ell$ gives the transport and diffusion identities
\[
\Omega^\ell\cdot (U^\ell\cdot\nabla)\Omega^\ell
=U^\ell\cdot\nabla\frac{|\Omega^\ell|^2}{2},
\qquad
-\Omega^\ell\cdot\Delta\Omega^\ell
=-\Delta\frac{|\Omega^\ell|^2}{2}+|\nabla\Omega^\ell|^2.
\]
Finally, the antisymmetric part of $\nabla U^\ell$ does not contribute to $\Omega^\ell\cdot(\Omega^\ell\cdot\nabla)U^\ell$, so
\[
\Omega^\ell\cdot(\Omega^\ell\cdot\nabla)U^\ell
=S^\ell\Omega^\ell\cdot\Omega^\ell.
\]
This proves \eqref{eq:filtered-enstrophy-identity}.
\end{proof}

\subsection{Scale-invariant filtered stretching quantities}

Let $\chi_{r,z_0}$ be a nonnegative cutoff supported in $Q_r(z_0)$ and equal to one on a slightly smaller cylinder, with standard parabolic derivative bounds. For $0<\ell\le c r$ define
\begin{align}
\mathsf O_{r,\ell}(z_0)&=r^{-1}\iint_{Q_r(z_0)}\chi_{r,z_0}|\Omega^\ell|^2\dxdt,\label{eq:filtered-O}\\
\mathsf P_{r,\ell}(z_0)&=r\iint_{Q_r(z_0)}\chi_{r,z_0}|\nabla\Omega^\ell|^2\dxdt,\label{eq:filtered-P}\\
\mathsf V^+_{r,\ell}(z_0)&=r\iint_{Q_r(z_0)}\chi_{r,z_0}(S^\ell\Omega^\ell\cdot\Omega^\ell)_+\dxdt,\label{eq:filtered-V}\\
\mathsf R_{r,\ell}(z_0)&=r\iint_{Q_r(z_0)}\chi_{r,z_0}|\Omega^\ell|\,|\mathcal J^\ell|\dxdt.\label{eq:filtered-R}
\end{align}
The cutoff terms generated from \eqref{eq:filtered-enstrophy-identity} are denoted by $\mathsf L_{r,\ell}(z_0)$. These quantities are scale invariant when the filter is chosen relatively, for example $\ell=\sigma r$.

A schematic direction-incoherence defect may be taken as follows. Let
\[
\xi^\ell=\frac{\Omega^\ell}{|\Omega^\ell|}
\]
where $\Omega^\ell\ne0$, with a harmless regularization on the zero set. One model defect is
\begin{equation}\label{eq:direction-defect}
\mathsf A_{r,\ell}(z_0)
=
\int_{t_0-r^2}^{t_0}\iint_{B_r\times B_r}
K_{r,\ell}(x,y)
|\xi^\ell(x,t)-\xi^\ell(y,t)|
|\Omega^\ell(x,t)|^2|\Omega^\ell(y,t)|\dx\,dy\,dt,
\end{equation}
where $K_{r,\ell}$ is a scale-normalized singular kernel chosen to match the Biot--Savart or strain representation. The precise kernel is part of the theorem design; the point is that strong stretching should force either diffusion or measurable directional decorrelation. This is aligned with the classical vorticity-direction philosophy of Constantin and Fefferman \cite{ConstantinFefferman1993}.

\subsection{The first hard lemma}

The first new PDE lemma to attack is the following.

\begin{problem}[First hard lemma: stretching--diffusion depletion]\label{prob:first-hard-lemma}
For normalized suitable weak solutions with $\Phi(z_0,r)\le M$ and for relative filters $\ell=\sigma r$, prove constants $\eps_*>0$ and $C(M)<\infty$ such that
\begin{equation}\label{eq:first-hard-lemma}
\mathsf V^+_{r,\ell}(z_0)
\le
(1-\eps_*)\mathsf P_{r,\ell}(z_0)
+C(M)\mathsf O_{r,\ell}(z_0)
+C\mathsf A_{r,\ell}(z_0)
+\mathsf R_{r,\ell}(z_0)
+\mathsf L_{r,\ell}(z_0).
\end{equation}
\end{problem}

The result should not be advertised as a regularity theorem. Its value is more basic: it would say that persistent positive stretching cannot be a silent source of badness. If \eqref{eq:first-hard-lemma} fails, the failure itself should identify a concrete obstruction class: high-frequency subgrid vorticity forcing, direction incoherence, caloric leakage, pressure-compatible transfer, or a coherent filtered Euler-like cascade.

\subsection{Stretching-dominant bad-scale counting}

Let $r_k=2^{-k}r_0$ and choose $\ell_k=\sigma r_k$. Define a stretching-dominant bad scale by
\begin{equation}\label{eq:stretching-bad-scale}
k\in\calB_{\rm str}
\quad\Longleftrightarrow\quad
\mathsf V^+_{r_k,\ell_k}
>
\lambda\mathsf P_{r_k,\ell_k}
+C(M)\mathsf O_{r_k,\ell_k}
+\mathsf D_{r_k,\ell_k},
\end{equation}
where $\mathsf D_{r_k,\ell_k}$ collects direction, subgrid, pressure-tail, and leakage defects.
The counting target is
\begin{equation}\label{eq:stretching-counting-target}
\sum_{k=0}^{N-1}
\left[
\mathsf V^+_{r_k,\ell_k}
-
\lambda\mathsf P_{r_k,\ell_k}
-
C(M)\mathsf O_{r_k,\ell_k}
\right]_+
\le
C(M)+
\sum_{k=0}^{N-1}\mathsf D_{r_k,\ell_k}.
\end{equation}
This is the corrected use of finite-chain language: count stretching-dominant scales, not abstract detector-bad scales.

A corresponding singularity alternative would be
\begin{equation}\label{eq:stretching-necessary-alternative}
(0,0)\text{ singular}
\quad\Longrightarrow\quad
\limsup_{k\to\infty}
\frac{\mathsf V^+_{r_k,\ell_k}}
{\mathsf P_{r_k,\ell_k}+\mathsf O_{r_k,\ell_k}+\mathsf D_{r_k,\ell_k}}
\ge c_0
\end{equation}
for some relative filter choices, or a positive-density dyadic version. This is a mechanism statement. It says that a true singularity must repeatedly pass through a stretching-dominant regime unless one of the explicit silent mechanisms carries the recurrence.

\subsection{How this reuses the existing calculus}

The previous finite-window machinery should now be used only downstream. Once a stretching-dominant slab, a direction-incoherent slab, or a subgrid-transfer slab has been identified, the detector calculus can quantify it, compare localized and clean representatives, and track leakage through a finite chain. But the primary object should be the PDE mechanism in \eqref{eq:filtered-enstrophy-identity}. The revised order is
\begin{equation}\label{eq:stretching-order}
\begin{gathered}
\text{filtered vorticity mechanism}
\Longrightarrow
\text{stretching--diffusion estimate}\\
\Longrightarrow
\text{bad-scale counting}
\Longrightarrow
\text{finite-window certification}.
\end{gathered}
\end{equation}

\section{From detector geometry to mechanism geometry}\label{sec:mechanism}

\subsection{Renormalized recurrence}

Let $r_n=2^{-n}$ and define the renormalized orbit
\[
(u^{(n)},p^{(n)})
=(r_nu(r_nx,r_n^2t),r_n^2p(r_nx,r_n^2t)).
\]
A persistent non-CKN branch is a recurrent orbit in a bad subset of a critical state space. The more intrinsic question is not which finite detector sees the orbit, but what nonlinear mechanism keeps rebuilding its badness. After the correction in \cref{sec:stretching-route}, the leading candidate mechanism is filtered vortex stretching; pressure work and resolved energy flux should be treated as companion channels rather than as the primary detector.

\subsection{Sustaining diagnostics}

For a relative filter scale $\ell\in(0,1)$ define the coarse flux activity
\begin{equation}\label{eq:flux-diagnostic}
F_{n,\ell}^{\rm abs}
=
\int_{Q_1}\chi|\Pi_{n,\ell}|\dxdt,
\qquad
F_{n,\ell}^{\rm sgn}
=
\abs{\int_{Q_1}\chi\Pi_{n,\ell}\dxdt},
\end{equation}
and the coherence ratio
\begin{equation}\label{eq:coherence-ratio}
\Gamma_{n,\ell}
=
\frac{F_{n,\ell}^{\rm sgn}}{F_{n,\ell}^{\rm abs}+\eps}.
\end{equation}
Split the pressure locally and define active pressure work
\begin{equation}\label{eq:pressure-diagnostic}
P_n^{\rm act}
=
\abs{\int_{Q_1}p_n^{\rm act}u^{(n)}\cdot\nabla\chi\dxdt}.
\end{equation}
For the filtered vorticity $\Omega_{n,\ell}=\nabla\times U_{n,\ell}$ and strain $S_{n,\ell}$ define
\begin{equation}\label{eq:vortex-diagnostic}
V_{n,\ell}^+
=
\int_{Q_1}\chi
\bigl(\Omega_{n,\ell}\cdot S_{n,\ell}\Omega_{n,\ell}\bigr)_+\dxdt,
\end{equation}
with an additional subgrid vorticity-transfer term if desired.

\begin{conjecture}[Necessary mechanism conjecture]\label{conj:mechanism}
If
\[
\Psi(r_n)\ge\eps_{\CKN}
\]
for all sufficiently large $n$, then there exist relative filter scales $\ell_n$ in a compact subinterval of $(0,1)$ and a subsequence such that
\begin{equation}\label{eq:mechanism-conjecture}
\limsup_{n\to\infty}
\left(
V_{n,\ell_n}^+
+P_n^{\rm act}
+F_{n,\ell_n}^{\rm abs}
\right)>0.
\end{equation}
\end{conjecture}

This conjecture is stronger in explanatory value than a detector lower bound. It asks for a source of recurrence. Flux-dominated recurrence then splits into sign-coherent and oscillatory cancellation regimes according to $\Gamma_{n,\ell_n}$. Pressure-dominated recurrence requires active, not merely harmonic, pressure focusing. Stretching-dominated recurrence asks whether persistent strain-vorticity alignment can coexist with filtered diffusion, subgrid vorticity transfer, pressure compatibility, and vorticity-direction coherence.

\subsection{Relationship with the existing detector calculus}

Detector geometry remains useful as a secondary language. Once a sustaining mechanism is identified, finite-window observations can quantify it, the work ledger can test whether it depletes energy, and quotient methods can separate physical activity from gauge and localization artifacts. The order should therefore be
\begin{equation}\label{eq:new-order}
\text{mechanism extraction}
\Longrightarrow
\text{PDE estimate}
\Longrightarrow
\text{finite-window certification},
\end{equation}
not the reverse.

\section{A conditional closure theorem}\label{sec:conditional-closure}

We now state a precise finite-chain theorem showing how the missing estimates would close the work-depletion route. The theorem is abstract but non-tautological: its hypotheses correspond directly to the new estimates identified above.

Let
\[
\calB_N=\{0\le k\le N-1:\Psi(r_k)\ge\eps_{\CKN}\}
\]
be the CKN-bad scales. Let $\calE_N\subset\calB_N$ be the exceptional silent scales, including unresolved residual, uncontrolled harmonic tail, or classified coherent profiles. Let $W_k^+$, $W_k^-$, $L_k$, $D_k$, and $w_k$ be the quantities in the fixed-chain combined-work depletion inequality \eqref{eq:work-depletion}.

\begin{assumption}[Forward extraction off the silent set]\label{ass:forward-extraction}
There is $c_0>0$ such that
\begin{equation}\label{eq:forward-extraction}
W_k^+\ge c_0
\qquad
\text{for every }k\in\calB_N\setminus\calE_N.
\end{equation}
\end{assumption}

\begin{assumption}[Silent-set budget]\label{ass:silent-budget}
There is $C_{\rm sil}<\infty$, independent of $N$, such that
\begin{equation}\label{eq:silent-budget}
\sum_{k\in\calE_N}w_k\le C_{\rm sil}.
\end{equation}
\end{assumption}

\begin{assumption}[Leakage and backscatter absorption]\label{ass:absorption}
There exist $0\le\gamma<1$ and $C_{\rm err}<\infty$, independent of $N$, such that
\begin{equation}\label{eq:absorption}
\sum_{k=0}^{N-1}w_k\bigl(|L_k|+W_k^-\bigr)
\le
\gamma\sum_{k\in\calB_N\setminus\calE_N}w_kW_k^+
+C_{\rm err}.
\end{equation}
\end{assumption}

\begin{theorem}[Conditional finite-chain CKN closure]\label{thm:conditional-closure}
Under Assumptions~\ref{ass:forward-extraction}--\ref{ass:absorption},
\begin{equation}\label{eq:bad-weight-bound}
\sum_{k\in\calB_N}w_k
\le
C_{\rm sil}
+
\frac{\calE_0+C_{\rm err}}{(1-\gamma)c_0}.
\end{equation}
Consequently, if the total scale weight satisfies
\begin{equation}\label{eq:divergent-total-weight}
\sum_{k=0}^{N-1}w_k\longrightarrow\infty
\qquad\text{as }N\to\infty,
\end{equation}
then an indefinitely persistent non-CKN branch is impossible: some sufficiently small scale is CKN-small.
\end{theorem}

\begin{proof}
The work-depletion inequality \eqref{eq:work-depletion} and \eqref{eq:absorption} give
\begin{align*}
\sum_{k\in\calB_N\setminus\calE_N}w_kW_k^+
&\le
\sum_{k=0}^{N-1}w_k(W_k^++D_k)\\
&\le \calE_0+
\sum_{k=0}^{N-1}w_k(|L_k|+W_k^-)\\
&\le \calE_0+
\gamma\sum_{k\in\calB_N\setminus\calE_N}w_kW_k^+
+C_{\rm err}.
\end{align*}
Hence
\[
(1-\gamma)
\sum_{k\in\calB_N\setminus\calE_N}w_kW_k^+
\le \calE_0+C_{\rm err}.
\]
Using \eqref{eq:forward-extraction},
\[
(1-\gamma)c_0
\sum_{k\in\calB_N\setminus\calE_N}w_k
\le \calE_0+C_{\rm err}.
\]
Add the silent-set bound \eqref{eq:silent-budget} to obtain \eqref{eq:bad-weight-bound}. If every sufficiently small scale were bad, then the left side would have the same divergent asymptotic weight as \eqref{eq:divergent-total-weight}, contradicting \eqref{eq:bad-weight-bound}.
\end{proof}

\begin{remark}[Why the theorem matters]
The theorem isolates the exact closure architecture:
\begin{align*}
\text{CKN badness}
&\xRightarrow{\text{extraction/classification}}
\text{forward work outside a controlled silent set},\\
&\xRightarrow{\text{leakage/backscatter absorption}}
\text{finite bad-scale weight}.
\end{align*}
The finite-window structures already provide the bookkeeping around these arrows. The new mathematics lies in proving the three hypotheses for an intrinsic NS-generated class.
\end{remark}

\begin{remark}[Choice of weights]
The weights $w_k=r_k/r_0$ in the fixed-slab work theorem have finite total mass on a geometric chain. Therefore \eqref{eq:divergent-total-weight} does not hold for those weights as $N\to\infty$. To derive an infinite-chain conclusion one needs either a renormalized weight system with divergent total mass, a positive-density unweighted extraction, or an interval formulation that yields a fixed amount of depletion per bad scale. This is another example of the distinction between a correct finite-chain theorem and a scale-uniform regularity mechanism.
\end{remark}

The last remark is important. The conditional theorem does not hide the weighting problem. It shows that even after detector extraction, one must design the chain normalization so that persistent badness creates a non-summable demand on a finite budget.

\section{Conclusion and outlook}\label{sec:conclusion-outlook}\label{sec:agenda}

This paper gives a structural audit of several finite-scale decompositions arising in the local regularity theory of the three-dimensional Navier--Stokes equations. The main conclusion is negative in a useful sense: the finite-window detector constructions considered here do not by themselves provide a coercive mechanism for excluding CKN badness. Their role is instead to separate resolved, pressure, flux, leakage, and subgrid contributions and to identify where a genuinely new PDE estimate would have to enter.

The components that remain mathematically useful are the finite-scale critical ledger, the separation between resolved and unresolved contributions, the fixed finite-window quotient estimate, the combined pressure--flux work identity, and the filtered vorticity formulation. These tools provide a precise bookkeeping framework for locating possible obstruction channels. They should not, however, be interpreted as a substitute for coercive control of the positive stretching term, the subgrid vorticity forcing, or the remaining leakage and backscatter channels.

The next step is therefore not to further refine the detector language, but to prove an estimate that acts directly on the sustaining mechanism of possible bad scales. Two natural directions are a localized filtered stretching--diffusion surplus estimate controlling the positive vortex-stretching contribution after pressure and flux transfer are accounted for, and a compactness--rigidity theorem showing that any minimal obstruction profile is incompatible with the Navier--Stokes local energy structure. Either result would turn the present audit framework from a decomposition theory into a genuine regularity or obstruction-exclusion mechanism.

\end{document}